\documentclass[reqno]{amsart}
\makeatletter\def\@biblabel#1{#1.}\makeatother
\usepackage{amsmath,amsfonts,amssymb,amsthm,mathrsfs,bbm,stmaryrd,array,titlesec,subfigure,fancynum}
\usepackage[colorlinks, citecolor=blue]{hyperref}
\newtheorem{lemma}{Lemma}
\newtheorem{conjecture}{Conjecture}
\newtheorem{theorem}{Theorem}

\newtheorem{corollary}{Corollary}
\newtheorem{definition}{Definition}

\setfnumgsym{\allowbreak}
\def\mydash{\CJKglue\raise0.2ex\hbox{---\kern-0.01em---}\CJKglue}
\def\NN{\mathbb{N}}
\def\ZZ{\mathbb{Z}}

\begin{document}
\title{Perfect numbers and Fibonacci primes }
\author{Tianxin Cai, Deyi Chen and Yong Zhang }
\date{}
\maketitle
\vspace{-3em}
\begin{abstract}
In this paper, we introduce the concept of $F$-perfect number, which is a positive integer $n$ such that $\sum_{d|n,d<n}d^2=3n$. We prove that all the $F$-perfect numbers are of the form $n=F_{2k-1}F_{2k+1}$, where both $F_{2k-1}$ and $F_{2k+1}$ are Fibonacci primes. Moreover, we obtain other interesting results and raise a new conjecture on perfect numbers.
\end{abstract}
\footnote[0]{


Project supported by the National Natural Science Foundation of China 11351002.}

\bigskip
\textbf{1. Introduction}
\smallskip

Many ancient cultures endowed certain integers with special religious and magical significance. One example is a perfect number $n$, which is equal to the sum of its proper positive divisors, i.e.,
\begin{equation*}
  \sum_{\substack{d|n\\d<n}}d=n.
\end{equation*}
 The first two perfect numbers are $6$ and $28$. It was suggested that God made the world in $6$ days because $6$ is a perfect number and the perfection of the universe was shown by moon's period of $28$ days. The next three perfect numbers are $496, 8128$, and $33550336$.

Euclid discovered an interesting connection between perfect numbers and Mersenne primes: If $M=2^p-1$ is a prime (Mersenne prime), then the $M$-th triangular number $T_{M}=\frac{1}{2}M(M+1)=2^{p-1}(2^p-1)$ is a perfect number. Almost two thousand years later, Euler showed that all the even perfect numbers are of Euclid's form. So far we only know $48$ Mersenne primes, hence we only know $48$ even perfect numbers. As for odd perfect numbers, it is unknown whether they exist. Also, no proof is known whether there are infinitely many perfect numbers. See \cite{1}: B1 for more unsolved problems on perfect numbers.

Meanwhile, some great mathematicians tried to find the generalization of perfect numbers, they consider
\begin{equation*}
  \sum_{\substack{d|n\\d<n}}d=kn,
\end{equation*}
where $k$ is a fixed positive integer, it's called $k$-perfect number. When $k=1$, it is ordinary perfect number. For $k=2$, Recorder found the first one $120$ in 1557, in that year he invented $``="$ as equal sign. Fermat found the second one $672$ in 1637, in that year he raised Fermat's Last Theorem. Andr$\acute{e}$ Jumeau found the third one $523776$. The fifth one $1476304896$ was found by Descartes in 1638, it was one year before Mersenne found the fourth one $459818240$. Descartes also found six others of class $3$: $30240, 32760, 23569920, 142990848, 66433720320$ and one of class $4$: $14182439040$. cf.\cite{4}. However, no one found any criterion for class $k>1$ as Euler did for class $1$. In modern time, quite a few number theorists including Lehmer and Carmichael studied this problem, they found thousands of multiply perfect numbers.

Let $\NN$ be the set of positive integers. In this paper, we first consider a generalization of perfect numbers:
\begin{equation}
  \sigma_{2}(n)-n^2=\sum_{\substack{d|n\\d<n}}{d^2}=bn,
\end{equation}
where $b \in \NN$.  We have
\begin{theorem}
  For any given positive integer $b\neq3$, the equation $(1)$ has only finitely many solutions. In particular, $(1)$ has no solution when $b=1,2$.
\end{theorem}
\begin{theorem}
For $b=3$, all the solutions of $(1)$ are $n=F_{2k-1}F_{2k+1}~(k\geq1)$, where  both $F_{2k-1}$ and $F_{2k+1}$ are  Fibonacci primes.
\end{theorem}

The original perfect numbers may be called $M$-perfect numbers because of Euler's result. In view of Theorem 1 and Theorem 2, we introduce the following definition.
\begin{definition}
  If a positive integer $n$ satisfies
  \begin{equation*}
    \sigma_{2}(n)-n^2=3n,
  \end{equation*}
  then we call $n$ an $F$-perfect number.
\end{definition}
As of November 2009, the largest known Fibonacci prime is $F_{81839}$ and the largest known probable Fibonacci prime is $F_{1968721}$ (it has $411,439$ digits). By the list of Fibonacci primes, we deduce that $n=F_{3}F_{5}=10, F_{5}F_{7}=65,F_{11}F_{13}=20737,$ $F_{431}F_{433}$ (180 digits), $F_{569}F_{571}$ (238 digits) are the only 5 $F$-perfect numbers we know. Note that $10=F_{3}F_{5}$ is the unique even $F$-perfect number. The next probable odd $F$-perfect number has at least $822,878$ digits. However, since the 36th Mersenne prime has $895,932$ digits and the 48th one has $17,425,170$ digits, if we use other computer system similar to GIMPS, which is used to find Mersenne prime, we may find more $F$-perfect numbers. Anyway, neither we know that there is any more odd $F$-perfect number nor if there are infinitely many $F$-perfect numbers.

Next, we consider a further generalization of perfect numbers:
\begin{equation}
  \sigma_{a}(n)-n^a=\sum_{\substack{d|n\\d<n}}{d^a}=bn,
\end{equation}
where integers $a\geq3$ and $b\geq1$. We obtain the following results.
\begin{theorem}
  For any given positive integers $a\geq3$ and $b\geq1$, the equation $(2)$ has only finitely many solutions.
\end{theorem}
\begin{corollary}
  For any given positive integer $a\geq2$, there are infinitely many positive integers $b$ such that $(2)$ has integer solutions.
\end{corollary}
\begin{theorem}
  If $n=pq$, where $p<q$ are primes and $n|\sigma_{3}(n)$, then $n=6$; if $n=2^\alpha p$ $(\alpha\geq1)$, $p$ is an odd prime and $n|\sigma_{3}(n)$, then $n$ is an even perfect number. The converse is also true except for $28$.
\end{theorem}
Moreover, let $\omega(n)$ be the number of prime factors of $n$ (with multiplicity), i.e., $\omega(n)=\sum_{p^\nu||n}{1}$. We have the following
\begin{conjecture}
  $\omega(n)=2$ and $n|\sigma_{3}(n)$ iff $n$ is an even perfect number except for $28$.
\end{conjecture}

\bigskip

\textbf{2. Preliminaries}
\begin{lemma}
  If $d>0$ is an odd integer, then $x^2-dy^2=-4$ has integer solutions iff $u^2-dv^2=-1$ has integer solutions.
\end{lemma}
\begin{proof}
  The Sufficiency is  obvious. Now we prove the necessity.\\
  Since $d>0$ is an odd integer and $x^2-dy^2=-4$, so $x\equiv y \pmod2$. Make the transformation
  \begin{equation*}
    \begin{cases}
      u=\frac{x(x^2+3)}{2},\\
      v=\frac{(x^2+1)y}{2}.
    \end{cases}
  \end{equation*}
 Then
  \begin{equation*}
    u^2-dv^2=\left(\frac{x(x^2+3)}{2}\right)^2-d\left(\frac{(x^2+1)y}{2}\right)^2=-1.
  \end{equation*}
\end{proof}
It is well known that if $N>0$ is not a perfect square,  the continued fraction representation of $\sqrt{N}$ has the form $\sqrt{N}=[a_0; \overline{a_1,a_2,\cdots,a_l}]$, where we denote $l(\sqrt{N})=l$ the period of the continued fraction of $\sqrt{N}$. We have
\begin{lemma}(see Kaplan and Williams \cite{2})
  If $N>0$ is not a perfect square, then $x^2-Ny^2=-1$ has integer solutions iff $l(\sqrt{N})$ is odd.
\end{lemma}
\begin{lemma}
  All solutions of the equation $1+x^2+y^2=3xy$~$(1\leq x<y)$ are
  \begin{equation}
   \begin{cases}
      x=F_{2k-1},\\
      y=F_{2k+1},
    \end{cases}
  \end{equation}
  where $k\geq1$ and $F_n$ is a Fibonacci number.
\end{lemma}
\begin{proof}
First of all, we prove that $x$ and $y$ given by $(3)$ satisfy the equation $1+x^2+y^2=3xy$. By Cassini's identity \cite{5}
\begin{equation*}
  F_n^2-F_{n+1}F_{n-1}=(-1)^{n-1},
\end{equation*}
we have
\begin{equation*}
  F_{2k}^2-F_{2k+1}F_{2k-1}=-1,
\end{equation*}
that is
\begin{equation*}
  (F_{2k+1}-F_{2k-1})^2-F_{2k+1}F_{2k-1}=-1.
\end{equation*}
Hence,
\begin{equation*}
  1+F_{2k-1}^2+F_{2k-1}^2=3F_{2k-1}F_{2k+1}.
\end{equation*}
Now we prove that $(3)$ gives all the solutions. We only need to prove that if $1+x^2+y^2=3xy$~$(1\leq x<y)$ then  $x=F_{2k-1}$. Note that
\begin{equation*}
  1+x^2+y^2=3xy\Longleftrightarrow5x^2-4=(3x-2y)^2,
\end{equation*}
by \cite{7} we deduce that $x$ is a Fibonacci number. If $x=F_{2k-1}$, we have done. If $x=F_{2k}$, then
\begin{equation*}
  5F_{2k}^2-4=(3x-2y)^2.
\end{equation*}
 By \cite{7}, last line of pp. 417, we have
\begin{equation*}
  5F_{2k}^2+4=L_{2k}^2,
\end{equation*}
where $L_0=2, L_1=1, L_{n+1}=L_n+L_{n-1}$ be the  Lucas  series. Hence
\begin{equation*}
  8=\left(5F_{2k}^2+4\right)-\left(5F_{2k}^2-4\right)=L_{2k}^2-(3x-2y)^2,
\end{equation*}
so $|3x-2y|=1$, $L_{2k}=3$ and $x=F_{2k}=\sqrt{\frac{L_{2k}^2-4}{5}}=1=F_2=F_1$.
Hence, we deduce that  $x=F_{2k-1}$.
\end{proof}
\begin{lemma}
  If a positive integer $k\neq3$, then $1+x^2+y^2=kxy$ has no solution.
\end{lemma}
\begin{proof}
  Since $kxy=1+x^2+y^2>2xy$, we only need to prove Lemma 4 for $k\geq4$. If $1+x^2+y^2=kxy$ has integer solutions, the discriminant of above
quadratic in $x$ must be a square, so there exists $z\in\ZZ$ such that
  \begin{equation*}
    k^2y^2-4(y^2+1)=(k^2-4)y^2-4=z^2,
  \end{equation*}
or
  \begin{equation*}
    z^2-(k^2-4)y^2=-4.
  \end{equation*}
  Now we show that $k$ is odd. Suppose  $k$ is even, then $x$ and $y$ can't be both even.
  \\If only one of $x$ and $y$ is odd, then
  \begin{equation*}
    1+x^2+y^2=kxy~~\pmod 4\Longrightarrow~~2\equiv0 \pmod4.
  \end{equation*}
  If $x$ and $y$ are both odd, then
  \begin{equation*}
    1+x^2+y^2=kxy~~\pmod 2\Longrightarrow~~1\equiv0 \pmod2.
  \end{equation*}
  Hence $k$ is odd.  Obviously, $k^2-4$ is not a perfect square when the odd integer $k\geq4$, so by Lemmas 1 and 2 we obtain
  \begin{align*}
    &z^2-(k^2-4)y^2=-4~has~integer~solutions\\
    \Longleftrightarrow & \; u^2-(k^2-4)v^2=-1~has~integer~solutions\\
    \Longleftrightarrow & \; l(\sqrt{k^2-4})~is~odd.
  \end{align*}
 For the odd integer $k\geq4$, by \cite{6}, pp. 503, Exercise 11, we have $\sqrt{k^2-4}=[k-1; \overline{1,\frac{k-3}{2},2,\frac{k-3}{2},1,2k-2}]$, i.e., $l(\sqrt{k^2-4})=6$. Contradiction.
\end{proof}
\begin{lemma}(See Luca and Ferdinands \cite{3})
  For any given positive integer $a\geq2$ there are infinitely many even numbers $n$ such that $n|\sigma_a(n)$.
\end{lemma}
\begin{lemma}
  If primes $p<q$ satisfy $p|q+1$ and $q|p+1$, then $p=2$ and $q=3$.
\end{lemma}
\begin{proof}
  It is obvious that primes $p<q$ satisfy $p|q+1$ and $q|p+1$ iff there exists a positive integer $k$ such that $1+p+q=kpq$. If $k=1$, then $1+p+q=pq$, i.e., $(p-1)(q-1)=2$, so $p=2,q=3$. If $k\geq2$, then $kpq\geq2pq>p+q+1$.
\end{proof}
\begin{lemma}
  If $\frac{x^2-x+1}{xy-1}\in\NN$,  where $x,y\in\NN$ and $y\geq2$, then $\frac{x^2-x+1}{xy-1}=1$.
\end{lemma}
\begin{proof}
  When $(x,y)=(1,2)$, we have $\frac{x^2-x+1}{xy-1}=1$. Suppose that $\frac{x^2-x+1}{xy-1}=n\geq2$. Then $x^2-(1+ny)x+n+1=0$, so the discriminant must be a square, i.e., there exists $z\in\ZZ$ such that
  \begin{equation*}
    (1+ny)^2-4(n+1)=z^2,
  \end{equation*}
or
  \begin{equation*}
    \begin{cases}
    1+ny+z=a,\\
    1+ny-z=b,
  \end{cases}
  \end{equation*}
  where $ab=4(n+1)$. Therefore, $2ny=a+b-2$, i.e.,
  \begin{equation*}
    y=\frac{a+b-2}{2n}=\frac{a+b-2}{2\times(\frac{ab}{4}-1)}=\frac{2a+2b-4}{ab-4}.
  \end{equation*}
  We may suppose that $a\geq b>0$. If $b\geq5$ then $y=\frac{2a+2b-4}{ab-4}\leq\frac{2a+2b-4}{5b-4}<1$. Hence, $b=1, 2, 3$, or $4$.\\
  If $b=1$, then $a=4(n+1)\geq12$, so $2<y=\frac{2a-2}{a-4}<3$.\\
  If $b=2$, then $a=2(n+1)\geq6$, so $1<y=\frac{a}{a-2}<2$.\\
  If $b=3$, then $a=\frac{4}{3}(n+1)\geq4$. When $a\geq6$, we have $y=\frac{2a+2}{3a-4}\leq1$; when $4\leq a\leq5$, we have $1<y=\frac{2a+2}{3a-4}<2$.\\
  If $b=4$, then $a=n+1\geq3$. When $a=3$, we have $y=\frac{a+2}{2a-2}=\frac{5}{4}$; when $a\geq4$, we have $y=\frac{a+2}{2a-2}\leq1$.
\end{proof}
\begin{lemma}
  If positive integers $x\geq y$ satisfy $x|y^2-y+1$ and $y|x^2-x+1$, then $x=y=1$.
\end{lemma}
\begin{proof}
  When $x=y$, we have $x=y=1$. Now suppose $x>y$. Then
  \begin{equation*}
    \begin{cases}
      y^2-y+1=xt_2,\\
      x^2-x+1=yt_1,
    \end{cases}
  \end{equation*}
  where $t_1, t_2 \in\NN$. When $t_2=y$, then $y|1$ and we have $x=y=1$. When $t_2>y$, we have $y^2-y+1=xt_2\geq(y+1)\cdot(y+1)=y^2+2y+1$, which is impossible. Hence $y>t_2$ and
  \begin{align*}
    &t_{2}^{2}yt_1\\
    =&t_{2}^{2}(x^2-x+1)\\
    =&(xt_2)^2-t_2(xt_2)+t_{2}^{2}\\
    =&(y^2-y+1)^2-t_2(y^2-y+1)+t_{2}^{2}\\
    =&(y^2-y)^2+2(y^2-y)-t_2(y^2-y)+t_{2}^{2}-t_2+1,
  \end{align*}
  so $y|t_{2}^{2}-t_2+1$, i.e., there exist $t_3\in\NN$ such that  $t_{2}^{2}-t_2+1=yt_3$. Note that $y^2-y+1=xt_2=t_2x$, we have
  \begin{equation*}
    \begin{cases}
      t_{2}^{2}-t_2+1=yt_3,\\
      y^2-y+1=t_2x,
    \end{cases}
  \end{equation*}
  where $y>t_2$. Repeating this process, we have $x>y>t_2>t_3>\cdots>t_k=1$ with the numbers $t_2,t_3,\cdots,t_n$ satisfying
  \begin{equation*}
    \begin{cases}
      t_{i+1}|t_{i}^{2}-t_{i}+1,\\
      t_{i}|t_{i+1}^2-t_{i+1}+1.
    \end{cases}
  \end{equation*}
  But $t_{k-1}|t_{k}^2-t_{k}+1$ and $t_k=1$, so $t_{k-1}=1$. Hence $x=y=t_2=\cdots=t_k=1$. Contradiction.
\end{proof}
\bigskip

\textbf{3. Proofs of the Theorems}
\begin{proof}[\textbf{Proof of Theorem 1} \\]
Let $n=p_{1}^{\alpha_1}p_{2}^{\alpha_2}\cdots p_{k}^{\alpha_k}$ be the prime factorization of $n$, where $p_1<p_2<\cdots<p_k$, $\alpha_i\geq1$,
and $k\geq1$. When $k=1$, we have $\sigma_2(n)-n^2=1+p_1^2+\cdots+p_1^{2(\alpha_1-1)}=bp_1^{\alpha_1}$, which is impossible. When $k\geq3$,
 we have
\begin{align*}
   bn&=\sigma_2(n)-n^2\\
    &>\frac{n_1^2}{p_1^2}+\frac{n_2^2}{p_2^2}+\cdots+\frac{n_k^2}{p_k^2}\\
    &\geq k\left(\frac{n_1^2}{p_1^2}\frac{n_2^2}{p_2^2}\cdots\frac{n_k^2}{p_k^2}\right)^{\frac{1}{k}} (by~using~the~arithmetic-geometric~mean~ inequality)\\
    &=\left(kp_1^{-\frac{2}{k}}p_2^{-\frac{2}{k}}\cdots p_k^{-\frac{2}{k}}\right)n^2\\
    &=\left(kp_1^{\alpha_1-\frac{2}{k}}p_2^{\alpha_2-\frac{2}{k}}\cdots p_k^{\alpha_k-\frac{2}{k}}\right)n.
\end{align*}
Note that $\alpha_i-\frac{2}{k}\geq \frac{\alpha_i}{3}$ for each $1\leq i\leq k$. We have
\begin{equation}
  b\geq kp_1^{\alpha_1-\frac{2}{k}}p_2^{\alpha_2-\frac{2}{k}}\cdots p_k^{\alpha_k-\frac{2}{k}}\geq k\prod_{i=1}^{k}p_i^{\frac{\alpha_i}{3}}\geq3n^{\frac{1}{3}}.
\end{equation}
Hence $n$ is bounded.\\
When $k=2$, by $(6)$ we have
\begin{equation}
 b\geq 2p_1^{\alpha_1-1}p_2^{\alpha_2-1},
\end{equation}
from which we obtain that $\alpha_i~(1\leq i\leq 2)$ are bounded. If $\alpha_2>1$, then we deduce from $(7)$ that $p_2$ is bounded, so $n$ is bounded. If $\alpha_2=1$ and $\alpha_1>1$, then $p_1$ is bounded by $(7)$. But since $\sigma_2(n)-n^2=bn$, i.e.,
\begin{equation*}
  (1+p_1^2+p_1^4+\cdots+p_1^{2\alpha_1})(1+p_2^2)-p_1^{2\alpha_1}p_2^2=bp_1^{\alpha_1}p_2,
\end{equation*}
we have
\begin{equation}
  (1+p_1^2+p_1^4+\cdots+p_1^{2\alpha_1})p_2^2-bp_1^{\alpha_1}p_2+(1+p_1^2+p_1^4+\cdots+p_1^{2\alpha_1})=0.
\end{equation}
It follows from $(8)$ that $p_2$ is bounded since $p_1$ is bounded. So $n$ is bounded.\\
If $\alpha_1=\alpha_2=1$, then by $\sigma_2(n)-n^2=bn$ we have
\begin{equation}
  1+p_2^2+p_2^2=bp_1p_2.
\end{equation}
When $b\neq3$, by Lemma 4 we obtain that $(9)$ has no integer solutions.\\
Finally, by $(6),(7)$ and $(9)$, it is easy to deduce that $(1)$ has no solution when $b=1,2$.
\end{proof}
\begin{proof}[\textbf{Proof of Theorem 2} \\]
Let $n=p_{1}^{\alpha_1}p_{2}^{\alpha_2}\cdots p_{k}^{\alpha_k}$ be the prime factorization of $n$, where $p_1<p_2<\cdots<p_k$, $\alpha_i\geq1$, $k\geq1$. Now we prove that if $n$ is an $F$-perfect number, then $k=2$ and $\alpha_1=\alpha_2=1$. If $k=1$ we have $\sigma_2(n)-n^2=1+p_1^2+p_1^4+\cdots+p_1^{2(\alpha_1-1)}=3p_1^{\alpha_1}$, which is impossible. If $k\geq3$, then by $(6)$ we have $n=1$, which is impossible. Thus $k=2$. We now show that $\alpha_1=\alpha_2=1$. Suppose $\alpha_1+\alpha_2\geq3$. Noting that $\alpha_1^2-\alpha_1+2\alpha_1\alpha_2\geq\alpha_1(\alpha_1+\alpha_2)$ and $\alpha_2^2-\alpha_2+2\alpha_1\alpha_2\geq\alpha_2(\alpha_1+\alpha_2)$, and  by using the arithmetic-geometric mean inequality, we have
\begin{align*}
  3n&=\sigma_2(n)-n^2\\
  &>(1+p_1^2+p_1^4+\cdots+p_1^{2(\alpha_1-1)})p_2^{2\alpha_2}+(1+p_2^2+p_2^4+\cdots+p_2^{2(\alpha_2-1)})p_1^{2\alpha_1}\\
  &>(\alpha_1+\alpha_2)\left(p_2^{2\alpha_2}\cdot p_1^2p_2^{2\alpha_2}\cdots p_1^{2(\alpha_2-1)}p_2^{2\alpha}\cdot p_1^{2\alpha_1}\cdot p_2^{2}p_1^{2\alpha_1}\cdots p_2^{2(\alpha_2-1)}p_1^{2\alpha_1}\right)^{\frac{1}{\alpha_1+\alpha_2}}\\
  &\geq 3 p_1^{\frac{\alpha_1^2-\alpha_1+2\alpha_1\alpha_2}{\alpha_1+\alpha_2}}p_2^{\frac{\alpha_2^2-\alpha_2+2\alpha_1\alpha_2}{\alpha_1+\alpha_2}}\\
  &\geq 3p_1^{\alpha_1}p_2^{\alpha_2}\\
  &=3n,
\end{align*}
which is impossible. Therefore, $k=2, \alpha_1=\alpha_2=1$ and $\sigma_2(n)-n^2=3n$, i.e.,
\begin{equation*}
  1+p_1^2+p_2^2=3p_1p_1.
\end{equation*}
Theorem 2 follows from Lemma 3.
\end{proof}
\begin{proof}[\textbf{Proof of Theorem 3} \\]
Let $n=p_{1}^{\alpha_1}p_{2}^{\alpha_2}\cdots p_{k}^{\alpha_k}$ be the prime factorization of $n$, where $p_1<p_2<\cdots<p_k$, $\alpha_i\geq1$, $k\geq1$. It is obvious that $k\neq1$.\\
When $k\geq2$, similar to  the proof of Theorem  1, we have
\begin{equation*}
  b\geq kp_1^{(a-1)\alpha_1-\frac{a}{k}}p_2^{(a-1)\alpha_2-\frac{a}{k}}\cdots p_k^{(a-1)\alpha_k-\frac{a}{k}}\geq k\prod_{i=1}^kp_i^{\frac{\alpha_i}{2}}\geq2n^{\frac{1}{2}}
\end{equation*}
for $a\geq3$. So $n$ is bounded.
\end{proof}

Corollary 1  follows from Theorems 1-3 and Lemma 5; noting that $(1)$ has only one even integer solution for $a=2,b=3$.

\begin{proof}[\textbf{Proof of Theorem 4} \\]
When $n=pq~(p<q)$, then $\sigma_3(n)=1+p^3+q^3+p^3q^3$. But since $n|\sigma_3(n)$, so $pq|1+p^3+q^3$, i.e., $p|q^3+1$ and $q|p^3+1$.
It follows from $x^3+1=(x+1)(x^2-x+1)$  that $p|q^3+1$ and $q|p^3+1$ are equivalent to
\begin{equation*}
  \begin{cases}
    p|q+1\\
    q|p+1
  \end{cases}
  or~~~~~
  \begin{cases}
    p|q+1\\
    q|p^2-p+1\\
  \end{cases}
  or~~~~~
  \begin{cases}
    q|p+1\\
    p|q^2-q+1\\
  \end{cases}
  or~~~~~
  \begin{cases}
    p|q^2-q+1\\
    q|p^2-p+1.\\
  \end{cases}
\end{equation*}
By Lemma 6, the system
\begin{equation*}
  \begin{cases}
    p|q+1,\\
    q|p+1 \\
  \end{cases}
\end{equation*}
has the only solution $(p,q)=(2,3)$.
By Lemma 8,
\begin{equation*}
  \begin{cases}
    p|q^2-q+1\\
    q|p^2-p+1\\
  \end{cases}
\end{equation*}
has no solution.\\
 If
 $\begin{cases}
    p|q+1\\
    q|p^2-p+1\\
  \end{cases}$,
  then $pq|(q+1)(p^2-p+1)$, i.e., $pq|p^2-p+q+1$, so there exists $k\in\NN$ such that $p^2-p+q+1=kpq$, i.e.,
  \begin{equation}
    q=\frac{p^2-p+1}{kp-1}\geq2.
  \end{equation}
  If $k=1$, then we have $q=\frac{p^2-p+1}{p-1}=p+\frac{1}{p-1}$, so $(p,q)=(2,3)$.
  If $k\geq2$, then by Lemma 7, $(10)$ has no solution.\\
  Similarly, we can obtain that $\begin{cases}
    q|p+1\\
    p|q^2-q+1\\
  \end{cases}$has no solution. This proves the first part of Theorem 4. \\
Now we prove the second part. First of all, we prove that if $n$ is an even perfect number except for $28$, then $n|\sigma_3(n)$. When $n$ is an even perfect number, we have $n=2^{p-1}(2^p-1)$, where $p$ and $2^p-1$ are primes. It is obvious that $6|\sigma_3(6)$ and $28\nmid\sigma_3(28)$. We suppose that $n=2^{p-1}(2^p-1)$, where prime $p\geq5$ and $2^p-1$ is prime. Noting that $2^{3p}-1=(2^p-1)(2^{2p}+2^p+1)$ and $2^{3p}-1=7(1+2^3+2^6+\cdots+2^{3(p-1)})$, so $7|2^{2p}+2^p+1$ and
\begin{align*}
  \sigma_3(n)&=\sigma_3(2^{p-1}(2^p-1))\\
  &=(1+2^3+2^6+\cdots+2^{3(p-1)})(1+(2^p-1)^3)\\
  &=(2^p-1)\frac{2^{2p}+2^p+1}{7}2^p\left((2^p-1)^2-(2^p-1)+1\right)\\
  &=n\frac{2^{2p}+2^p+1}{7}\left((2^p-1)^2-(2^p-1)+1\right),
\end{align*}
from which we deduce that if $n$ is an even perfect number except for $28$, then $n|\sigma_3(n)$.\\
Secondly, we prove that if $n=2^{\alpha-1} p$, where $p$ is an odd prime, $\alpha\geq2$ and $n|\sigma_3(n)$, then $\alpha$ is prime and $p=2^{\alpha}-1$. Noting that
\begin{align*}
  \sigma_3(n)&=\sigma_3(2^{\alpha-1}p)\\
  &=\sigma_3(2^{\alpha-1})\sigma_3(p)\\
  &=(1+2^3+2^6+\cdots+2^{3(\alpha-1)})(1+p^3)\\
  &=(1+2^3+2^6+\cdots+2^{3(\alpha-1)})(1+p)(1-p+p^2)\\
  &\equiv0 \pmod{2^{\alpha-1} p},
\end{align*}
we have $1+p\equiv0 \pmod{2^{\alpha-1}}$ and $1+2^3+2^6+\cdots+2^{3(\alpha-1)}\equiv0 \pmod{p}$, i.e., there exist $k_1, k_2\in\NN$ such that $p=k_12^{\alpha-1}-1$  and $1+2^3+2^6+\cdots+2^{3(\alpha-1)}=\frac{2^{3\alpha}-1}{7}=k_2p$. So
\begin{equation}
  2^{3\alpha}-1=(2^\alpha-1)(2^{2\alpha}+2^\alpha+1)=k_3(k_12^{\alpha-1}-1),
\end{equation}
where $k_3=7k_2$.\\
If $k_1=1$, then we have $p=2^{\alpha-1}-1$, and  it follows from $(11)$ that  $2^\alpha-1 \equiv0 \pmod{2^{\alpha-1}-1}$ or $2^{2\alpha}+2^\alpha+1\equiv0 \pmod{2^{\alpha-1}-1}$. When $2^\alpha-1 \equiv0 \pmod{2^{\alpha-1}-1}$, we have
\begin{equation*}
  1=2^\alpha-1-2(2^{\alpha-1}-1)\equiv0 \pmod{2^{\alpha-1}-1},
\end{equation*}
which is impossible.
When $2^{2\alpha}+2^\alpha+1\equiv0 \pmod{2^{\alpha-1}-1}$, we have
\begin{equation*}
  0\equiv 2^{2\alpha}+2^{\alpha}+1=(2^{\alpha-1}-1)(2^{\alpha+1}+6)+7\equiv7 \pmod{2^{\alpha-1}-1},
\end{equation*}
so $\alpha=4$. Then $n=2^{\alpha-1}p=2^{3}(2^{3}-1)=56$ and $n|\sigma_3(n)$, which is impossible.\\
If $k_1\geq3$, then by $(11)$ we have
\begin{equation*}
  2^{2\alpha}+2^\alpha+1\equiv0 \pmod{k_12^{\alpha-1}-1}.
\end{equation*}
There exists $k_4\in\NN$ such that
\begin{equation}
  2^{2\alpha}+2^\alpha+1=k_4(k_12^{\alpha-1}-1).
\end{equation}
So
\begin{equation*}
  1\equiv-k_4 \pmod{2^{\alpha-1}},
\end{equation*}
there exists $k_5\in\NN$ such that
\begin{equation}
  k_4=k_52^{\alpha-1}-1.
\end{equation}
Combining $(12)$ with $(13)$,
\begin{equation*}
  2^{2\alpha}+2^\alpha+1=(k_52^{\alpha-1}-1)(k_12^{\alpha-1}-1),
\end{equation*}
hence
\begin{equation*}
  2^{\alpha-1}=\frac{2+k_1+k_5}{k_1k_5-4}\geq2,
\end{equation*}
or
\begin{equation}
  (k_1-1)(k_5-1)+k_1k_5\leq11,
\end{equation}
from which we deduce that $k_5=1$ or $2$ since $k_1\geq3$. \\
When $k_5=1$, by $(14)$, $k_1\leq11$, so $3\leq k_1\leq11$. Noting that $2^{\alpha-1}=\frac{2+k_1+k_5}{k_1k_5-4}=\frac{k_1+3}{k_1-4}$, so $k_1=5$, $\alpha=4$, and the prime number $p=k_12^{\alpha-1}-1=39$, which is impossible.\\
When $k_5=2$, by $(14)$, $k_1\leq4$, so $k_1=3$ or $4$. Noting that $2^{\alpha-1}=\frac{2+k_1+k_5}{k_1k_5-4}=\frac{k_1+4}{2k_1-4}$, so $k_1=4$ and $\alpha=2$, $n=2^{\alpha-1}p=2^{\alpha-1}(k_12^{\alpha-1}-1)=14$, and $n|\sigma_3(n)$, which is impossible.\\
Hence $k_1=2$, from which prime $p=k_12^{\alpha-1}-1=2^\alpha-1$ and  $\alpha$ is prime.
\end{proof}


\bigskip
\noindent{Department of Mathematics, Zhejiang University, Hangzhou, 310027, China}\\
\textbf{Email address: txcai$@$zju.edu.cn}
\bigskip

\noindent{Department of Mathematics, Zhejiang University, Hangzhou, 310027, China}\\
\textbf{Email address: chendeyi1986$@$126.com}
\bigskip

\noindent{Department of Mathematics, Zhejiang University, Hangzhou, 310027, China}\\
\textbf{Email address: zhangyongzju@163.com}
\end{document}